\documentclass[envcountsect]{llncs}
\usepackage[utf8]{inputenc}
\usepackage{amssymb,amsmath,mathrsfs}
\usepackage[english]{babel}
\usepackage[unicode,unicode=true,bookmarks=false]{hyperref}

\newcommand{\sco}{,\ldots,}

\newcommand{\NN}{\mathbb{N}}
\newcommand{\ZZ}{\mathbb{Z}}
\newcommand{\QQ}{\mathbb{Q}}

\newcommand{\PrA}{\mathop{\mathbf{PrA}}\nolimits}
\newcommand{\PA}{\mathop{\mathbf{PA}}\nolimits}
\newcommand{\ZF}{\mathop{\mathbf{ZF}}\nolimits}
\newcommand{\Th}{\mathop{\mathbf{Th}}\nolimits}

\makeatletter

\renewcommand{\epsilon}{\varepsilon}
\renewcommand{\phi}{\varphi}
\newcommand{\sref}[2]{\hyperref[#2]{#1 \ref*{#2}}}
\newcommand{\dref}[2]{\hyperref[#2]{ #1 }}

\newcommand{\Lc}{\mathcal{L}}

\newcommand{\eqdef}{\stackrel{\mbox{\tiny\rm def}}{=}}\renewcommand{\to}{\rightarrow}
\newcommand{\ra}{\rightarrow}
\newcommand{\Ra}{\Rightarrow}

\spnewtheorem{hyp}{Conjecture}[section]{\bfseries}{\itshape}

\newcommand{\xv}{\overline{x}}
\newcommand{\yv}{\overline{y}}
\newcommand{\av}{\overline{a}}
\newcommand{\bv}{\overline{b}}
\newcommand{\vv}{\overline{v}}
\newcommand{\uv}{\overline{u}}
\newcommand{\cv}{\overline{c}}

\begin{document}
\author{Alexander Zapryagaev
  \and Fedor Pakhomov\thanks{This work is supported by the Russian Science Foundation under grant 16-11-10252.} }
\title{Interpretations of Presburger Arithmetic in Itself}
\institute{Steklov Mathematical Institute of Russian Academy of Sciences, 8, Gubkina Str., Moscow, 119991, Russian Federation}

\maketitle

\begin{abstract}
Presburger arithmetic $\PrA$ is the true theory of natural numbers with addition. We study interpretations of $\PrA$ in itself. We prove that all one-dimensional self-interpretations are definably isomorphic to the identity self-interpretation. 
In order to prove the results we show that all linear orders that are interpretable in $(\mathbb{N},+)$ are scattered orders with the finite Hausdorff rank and that the ranks are bounded in terms of the dimension of the respective interpretations. From our result about self-interpretations of $\PrA$ it follows that $\PrA$ isn't one-dimensionally interpretable in any of its finite subtheories. We note that the latter was conjectured by A.~Visser.
\end{abstract}

\section{Introduction}
Presburger Arithmetic $\PrA$ is the first-order theory of natural numbers with addition. It was introduced by M.~Presburger in 1929 \cite{presburger}. Presburger Arithmetic is complete, recursively-axiomatizable, and decidable.

The method of interpretations is a standard tool in model theory and in the study of decidability of first-order theories \cite{tarskimostowski,hodges}. An interpretation of a theory $\mathbf{T}$ in a theory $\mathbf{U}$ essentially is a uniform first-order definition of models of $\mathbf{T}$ in models of $\mathbf{U}$ (we present a detailed definition in Section~3). In the paper we study certain questions about interpretability for Presburger Arithmetic that were well-studied in the case of stronger theories like Peano Arithmetic $\PA$. Although, from technical point of view the study of interpretability for Presburger Arithmetic uses completely different methods than the study of interpretability for $\PA$ (see for example \cite{visser}), we show that from interpretation-theoretic point of view, $\PrA$ has certain similarities to strong theories that prove all the instances of mathematical induction in their own language, i.e. $\PA$, Zermelo-Fraenkel set theory $\ZF,$ etc. 

A {\em reflexive} arithmetical theory (\cite[p.\,13]{visser}) is a theory that can prove the consistency of all its finitely axiomatizable subtheories. Peano Arithmetic $\PA$ and Zermelo-Fraenkel set theory $\ZF$ are among well-known reflexive theories. In fact, all sequential theories (very general class of theories similar to $\PA,$ see \cite[III.1(b)]{hajekpudlak}) that prove all instances of induction scheme in their language are reflexive. For sequential theories reflexivity implies that the theory cannot be interpreted in any of its finite subtheories.  A.~Visser have conjectured that this purely interpretational-theoretic property holds for $\PrA$ as well. Note that $\PrA$ satisfies full-induction scheme in its own language but cannot formalize the statements about consistency of formal theories.

The conjecture was studied by J.~Zoethout \cite{jetze}. Note that Presburger Arithmetic, unlike sequential theories, cannot encode tuples of natural numbers by single natural numbers. And hence for interpretations in Presburger Arithmetic it is important whether individual objects are interpreted by individual objects (one-dimensional interpretations) or by tuples of objects of some fixed length $m$ ($m$-dimensional interpretations). Zoethout considered only the case of one-dimensional interpretations and proved that if any one-dimensional interpretation of $\PrA$ in $(\NN,+)$ gives a model that is definably isomorphic to $(\mathbb{N},+)$ then Visser's conjecture holds for one-dimensional interpretations, i.e. there are no one-dimensional interpretations of $\PrA$ in its finite subtheories. In the present paper we show that the following theorem holds and thus prove Visser's conjecture for one-dimensional interpretations:

\begin{theorem}\label{1a}
For any model $\mathfrak{A}$ of $\PrA$  that is one-dimensionally interpreted in the model $(\NN,+)$, (a) $\mathfrak{A}$ is isomorphic to $(\NN,+)$; (b) the isomorphism is definable in $(\NN,+)$.
\end{theorem}

Note that \sref{Theorem}{1a}(a) was established by J.~Zoethout in \cite{jetze}.

We also study whether the generalization of \sref{Theorem}{1a} to multi-dimensional interpretations holds. We prove:

\begin{theorem}\label{lb}
For any $m$ and model $\mathfrak{A}$ of $\PrA$  that is $m$-dimensionally interpreted in $(\NN,+)$, the model $\mathfrak{A}$ is isomorphic to $(\NN,+)$.
\end{theorem}

We don't know whether the isomorphism is always definable in $(\NN,+)$.

In order to prove \sref{Theorem}{lb}, we show that for every $m$ each linear order that is $m$-dimensionally interpretable in $(\mathbb{N},+)$ is \emph{scattered}, i.e. it doesn't contain a dense suborder. Moreover, our construction gives an estimation for Cantor-Bendixson ranks of the orders (a notion of Cantor-Bendixson rank for scattered linear orders goes back to Hausdorff \cite{hausdorff} in order to give more precise estimation we use slightly different notion of $VD_*$-rank from \cite{krs}):

\begin{theorem}\label{ordering}
All linear orders $m$-dimensionally interpretable in $(\NN,+)$ have the $VD_*$-rank at most $m.$
\end{theorem}

Note that since every structure interpretable in $(\NN,+)$ is automatic, the fact that both the $VD_*$ and Hausdorff ranks of any scattered linear order interpretable in $(\mathbb{N},+)$ is finite follows from the results on automatic linear orders by B.~Khoussainov, S.~Rubin, and F.~Stephan \cite{krs}.


The work is organized as follows. Section 2 introduces the basic notions. In Section 3 we give the definitions of non-parametric interpretations and definable isomorphism of interpretations. In Section 4 we define the dimension of Presburger sets and prove \sref{Theorem}{ordering}. In Section 5 we prove \sref{Theorem}{1a} and explain how it implies the impossibility to interpret $\PrA$ in its finite subtheories. In Section 6 we discuss the approach for the multi-dimensional case. 


\section{Presburger Arithmetic and Definable Sets}

In the section we give some results about Presburger Arithmetic and definable sets in $(\mathbb{N},+)$ from the literature that will be relevant for our paper.

\begin{definition}
{\em Presburger Arithmetic} $(\PrA)$ is the elementary theory of the model $(\NN,+)$ of natural numbers with addition.
\end{definition}

It is easy to see that every number $n\in \mathbb{N}$, the relations $<$ and $\le,$ modulo comparison relations $\equiv_n$, for natural $n\ge 1$, and the functions $x\longmapsto nx$ of multiplication by a natural number $n$ are definable in the model $(\mathbb{N},+)$. We fix some definitions for these constants, relations, and functions. This gives us a translation from the first-order language $\Lc$ of the signature $\langle =,\{n\mid n\in\mathbb{N}\},+,<~,\{\equiv_n\mid n\ge 1\}, \{x\longmapsto n x\mid n\in \mathbb{N}\}\rangle$ to the first-order language $\Lc^{-}$ of the signature $\langle =,+\rangle$. Since $\PrA$ is the elementary theory of $(\mathbb{N},+)$, regardless of the choice of the definitions, the translation is uniquely determined up to $\PrA$-provable equivalence. Thus we could freely switch between $\Lc$-formulas and equivalent $\Lc^{-}$-formulas. Note that $\PrA$ admits the quantifier elimination in the extended language $\mathcal{L}$ \cite{presburger}.






%

The well-known fact about order types of nonstandard models of $\mathrm{PA}$ also holds for models of Presburger arithmetic:

\begin{theorem}\label{models-class}
Any nonstandard model $\mathfrak{A} \models\PrA$ has the order type $\NN+\ZZ\cdot A$, where $\langle A,<_A\rangle$ is some dense linear order without endpoints. Thus, in particular, any countable model of $\PrA$ either has the order type $\NN$ or $\NN+\ZZ\cdot\QQ.$
\end{theorem}


%

For vectors $\overline{c},\overline{p_1},\ldots,\overline{p_n}\in\ZZ^m$ we call the set $\{\overline{c}+\sum k_i \overline{p_i}\mid k_i\in\NN\}$ a {\em lattice} with the {\em generating} vectors $\overline{p_1},\ldots,\overline{p_n}$ and the \emph{initial} vector $\overline{c}$. If $\overline{p_1},\ldots,\overline{p_n}$ are linearly independent ($n\le m$) we call the set an {\em $n$-dimensional fundamental lattice}.

R.~Ito \cite{ir} have proved that any  union of finitely many (possibly, intersecting) lattices in $\NN^m$ is a disjoint union of finitely many fundamental lattices. S.~Ginsburg and E.~Spanier \cite[Theorem 1.3]{ginsburg} have shown that the subsets of $\NN^k$ definable in $(\mathbb{N},+)$ are exactly the subsets of $\NN^k$ that are unions of finitely many (possibly, intersecting) lattices; note that the sets from the latter class are known as \emph{semilinear} sets. Combining these two results we obtain

\begin{theorem} \label{fund}
All subsets of $\NN^k$ definable in $(\mathbb{N},+)$ are exactly the subsets of $\NN^k$ that are disjoint unions of finitely many fundamental lattices.
\end{theorem}

Let us now consider the extension of the first-order predicate language with an additional quantifier $\exists^{=y}x,$ called a {\em counting quantifier} (notion introduced in \cite{barrington}), used as follows: if $f(\overline{x},z)$ is an $\Lc$-formula with the free variables $\overline{x},z,$ then $F=\exists^{=y}z\:G(\overline{x},z)$ is also a formula with the free variables $\overline{x},y.$

We extend the standard assignment of truth values to first-order formulas in the model $(\mathbb{N},+)$ to formulas with counting quantifiers. For a formula $F(\overline{x},y)$ of the form $\exists^{=y}z\:G(\overline{x},z)$, a vector of natural numbers $\overline{a}$, and a natural number $n$ we say that $F(\overline{a},n)$ is true iff there are exactly $n$ distinct natural numbers $b$ such that $G(\overline{a},b)$ is true. H.~Apelt \cite{apelt} and N.~Schweikardt \cite{schweikardt} have discovered that such an extension does not extend the expressive power of $\PrA:$

\begin{theorem}(\cite[Corollary 5.10]{schweikardt})\label{unti}
Every $\Lc$-formula $F(\overline{x})$ that uses counting quantifiers is equivalent in $(\NN,+)$ to a quantifier-free $\Lc$-formula.
\end{theorem}

\section{Interpretations}



\begin{definition} \label{interpretation} Suppose we have two first-order signatures $\Omega_1$ and $\Omega_2$. An $m$-dimensional \emph{translation} $\iota$ of a first order language of the signature $\Omega_1$ to the first-order language of the signature $\Omega_2$ consists of
  \begin{enumerate}
  \item a first-order formula $\mathit{Dom}_{\iota}(\overline{y})$ of the signature $\Omega_2$, where $\overline{x}$ is a vector of variables of the length $m$, with the intended meaning of the definition of the domain of translation;
  \item  first-order formulas $\mathit{Pred}_{\iota,P}(\overline{y}_1,\ldots,\overline{y}_n)$ of the signature $\Omega_2$, where each $\overline{y}_i$ is a vector of variables of the length $m$, for each predicate $P(x_1,\ldots,x_n)$ from $\Omega_1$ (including $x_1=x_2$);
  \item first-order formulas $\mathit{Fun}_{\iota,f}(\overline{y}_0,\overline{y}_1,\ldots,\overline{y}_n,)$ of the signature $\Omega_2$, where each $\overline{y}_i$ is a vector of variables of the length $m$, for each function $f(x_1,\ldots,x_n)$ from $\Omega_1$.
  \end{enumerate}
  Translation $\iota$ is an {\em interpretation} of a model $\mathfrak{A}$ of the signature $\Omega_1$ with the domain $A$  in  a model $\mathbf{B}$ of the signature $\Omega_2$ with the domain $B$ if
  \begin{enumerate}
  \item $\mathit{Dom}_{\iota}(\overline{y})$ defines a non-empty subset $D\subseteq B^m$;
  \item $\mathit{Pred}_{\iota,=}(\overline{y}_1,\overline{y}_2)$ defines an equivalence relation $\sim$ on the set $D$;
  \item there is a bijection $h\colon  D/{\sim}\to A$ such that for each predicate $P(x_1,\ldots,x_n)$ from $\Omega_1$ and $\overline{b}_1,\ldots,\overline{b}_n\in D$ we have
    $$\mathfrak{A}\models P(h([\overline{b}_1]_{\sim}),\ldots,h([\overline{b}_n]_{\sim}))\iff \mathfrak{B}\models \mathit{Pred}_{\iota,P}(\overline{b}_1,\ldots,\overline{b}_n)$$
    and for each function $f(x_1,\ldots,x_n)$ from $\Omega_1$ and $\overline{b}_0,\overline{b}_1,\ldots,\overline{b}_n\in D$ we have
    $$\mathfrak{A}\models h([\overline{b}_0]_{\sim})=f(h([\overline{b}_1]_{\sim}),\ldots,h([\overline{b}_n]_{\sim}))\iff \mathfrak{B}\models \mathit{Fun}_{\iota,f}(\overline{b}_0,\overline{b}_1,\ldots,\overline{b}_n).$$
  \end{enumerate}
  Translation $\iota$ is an interpretation of a theory $\mathbf{T}$ of the signature $\Omega_1$ in a model $\mathfrak{B}$ of the signature $\Omega_2$ if it is an interpretation of some model of $\mathbf{T}$ in $\mathfrak{B}$.  $\iota$ is an interpretation of a theory $\mathbf{T}$ of the signature $\Omega_1$ in a theory $\mathbf{U}$ of the signature $\Omega_2$ if it is an interpretation of $\mathbf{T}$ in every model $\mathfrak{B}$ of $\mathbf{U}$.

  Translation $\iota$ is called \emph{non-relative} if the formula $\mathit{Dom}_{\iota}(\overline{y})\equiv\top$, where $\overline{y}$ is $(y_1,\ldots,y_m)$.  We say that translation $\iota$ has \emph{absolute equality} if the formula $\mathit{Pred}_{\iota,=}(\overline{y},\overline{z})$ is $y_1=z_1\land\ldots\land y_m=z_m$, where $\overline{y}$ is $(y_1,\ldots,y_m)$ and $\overline{z}$ is $(z_1,\ldots,z_m)$.
\end{definition}

Note that naturally for each translation $\iota$ of a signature $\Omega_1$ to a signature $\Omega_2$, we could define a map $F(x_1,\ldots,x_n)\longmapsto F^{\iota}(\overline{y}_1,\ldots,\overline{y}_m)$ from formulas of the signature $\Omega_1$ to formulas of the signature $\Omega_2$ such that if $\iota$ is an interpretation of a model $\mathfrak{A}$ in a model $\mathfrak{B}$ then for each  $\overline{b}_1,\ldots,\overline{b}_n\in D$ we have
$$\mathfrak{A}\models F(h([\overline{b}_1]_{\sim}),\ldots,h([\overline{b}_n]_{\sim}))\iff \mathfrak{B}\models \mathit{F}^{\iota}(\overline{b}_1,\ldots,\overline{b}_n),$$
where $m$, $D$, and $h$ are as in the definition above.

Also we note that if $\iota$ is an interpretation of a theory $\mathbf{T}$ in a model $\mathfrak{B}$ then there is a unique up to isomorphism model $\mathfrak{A}$ of $\mathbf{T}$ such that $\iota$ is an interpretation of $\mathfrak{B}$ in $\mathfrak{A}$.

\begin{definition}
  Suppose $\iota_1$ and $\iota_2$ are respectively an $m_1$-dimensional and $m_2$-dimensional translations from a signature $\Omega_1$ to a signature $\Omega_2$.  And suppose that $I(\overline{y},\overline{z})$ is a first-order formula of the signature $\Omega_2$, where $\overline{y}$ consists of $m_1$ variables and $\overline{z}$ consists of $m_2$ variables.

  Now assume $\iota_1$ and $\iota_2$ are interpretations of the same model $\mathfrak{A}$ of the signature $\Omega_1$ with the domain $A$ in a model $\mathfrak{B}$ of the signature $\Omega_2$ with the domain $B$. As in Definition~\ref{interpretation} translations $\iota_1$ and $\iota_2$ give us respectively sets $D_1\subseteq B^{m_1}$, $D_2\subseteq B^{m_2}$ and equivalence relations $\sim_1$ on $D_1$ and $\sim_2$ on $D_2$. Under this assumption we say that $I(\overline{y},\overline{z})$ is a \emph{definition of an isomorphism} of $\iota_1$ and $\iota_2$ if we could choose bijections $h_1\colon D_1\to A$ and $h_2\colon D_2\to A$ (satisfying properties of $h$ from Definition~\ref{interpretation}, for respective $\iota_i$) such that for each $\overline{b}\in D_1$ and $\overline{c}\in D_2$ we have
  $$h_1([\overline{b}]_{\sim_1})=h_2([\overline{c}]_{\sim_1})\iff \mathfrak{B}\models I(\overline{b},\overline{c}).$$

  If $\iota_1$ and $\iota_2$ are interpretations of the theory $\mathbf{T}$ in a theory $\mathbf{U}$ and  for each model $\mathfrak{B}$ of $\mathbf{U}$ the formula $I(\overline{y},\overline{z})$ is a definition of an isomorphism between $\iota_1$ and $\iota_2$ as interpretations in $\mathfrak{B}$m then we say that $I(\overline{y},\overline{z})$ is a \emph{definition of an isomorphism} between $\iota_1$ and $\iota_2$ as interpretations of  $\mathbf{T}$ in $\mathbf{U}$.

  If $\iota_1$ and $\iota_2$ are interpretations of a theory $\mathbf{T}$ in a theory $\mathbf{U}$ (a model $\mathfrak{A}$) and there is a definition of an isomorphism then we say that $\iota_1$ and $\iota_2$ as interpretations of a theory $\mathbf{T}$ in a theory $\mathbf{U}$ (a model $\mathfrak{A}$) are \emph{definably isomorphic}.
\end{definition}





Since the theory $\PrA$ that we study is an elementary theory of some model ($\PrA=\Th(\mathbb{N},+)$), actually there is not much difference between interpretations in the standard model and in the theory. A translation $\iota$ is an interpretation of some theory $\mathbf{T}$ in $\PrA$ iff $\iota$ is an interpretation of $\mathbf{T}$ in $(\NN,+)$. A formula $I$ is a definition of an isomorphism between interpretations $\iota_1$ and $\iota_2$ of some theory $\mathbf{T}$ in $\PrA$ iff $I$ is a definition of an isomorphism between $\iota_1$ and $\iota_2$ as interpretations of $\mathbf{T}$ in $(\NN,+)$.


\section{Linear Orders Interpretable in $(\NN,+)$}

\subsection{Functions Definable in Presburger Arithmetic}

%

\begin{definition}
Suppose $A\subseteq \NN^n$ is a definable set. We call a function $f\colon A\ra\NN$ \emph{piecewise polynomial} of a degree $\le m$ if there is a decomposition of $A$ into finitely many fundamental lattices $C_1,\ldots,C_k$ such that the restriction of $f$ on each $C_i$ is a polynomial with rational coefficients of a degree $\le m$ \footnote{In our work, we use the word `piecewise' only in the sense defined here.}.
\end{definition}

In particular, a {\em piecewise linear} function is a piecewise polynomial function of a degree $\le 1$.

%

\begin{theorem}\label{jetz}
All definable in $(\NN,+)$ functions $f\colon\NN^n\ra\NN$ are exactly piecewise linear.
\end{theorem}
\begin{proof}
The definability of all piecewise linear functions in Presburger Arithmetic is obvious. A function $f\colon \NN^n\to \NN$ is definable iff its graph

\begin{center}
$G=\{(f(a_1,\ldots,a_n),a_1,\ldots,a_n)\mid (a_1,\ldots,a_n)\in \NN^n\}$
\end{center}
\noindent
is definable. According to \sref{Theorem}{fund}, $G$ is a finite union of fundamental lattices $J_1\sqcup\ldots \sqcup J_k$. For $1\le i\le k$ we denote by $J_i'$ the projections of $J_i$ along the first coordinate, $J_i'= \{(a_1,\ldots,a_n)\mid\exists a_0 ( (a_0,a_1,\ldots,a_n)\in J_i)\}$. Clearly, all $J_i'$ are fundamental lattices. And the restriction of the function $f$ on each of $J_i'$ is linear.
\end{proof}

\begin{corollary}\label{irration}
All definable in $(\NN,+)$ functions $f\colon \NN\to\NN$ can be bounded from above by a linear function with a rational slope. Conversely, if $h_1(x)<f(x)<h_2(x)$ for all $x,$ where $h_{1}(x)$ and $h_{2}(x)$ are linear functions of the same irrational slope, then $f(x)$ is not definable.
\end{corollary}

\subsection{Dimension}

Here we give the definition for the notion of dimension of Presburger-definable sets.

\begin{definition}
The {\em dimension} $\dim(A)$ of a Presburger-definable set $A\subseteq\NN^m$ is defined as follows.

\begin{itemize}
\item $\dim(A)=0$ iff $A$ is empty or finite;

\item $\dim(A)=k\ge 1$ iff there is a definable bijection between $A$ and $\NN^k.$
\end{itemize}
\end{definition}

The following theorem shows that the definition indeed gives the unique dimension for each $\PrA$-definable set.

\begin{theorem}\label{Existe}
Suppose $M$ is an infinite Presburger definable subset of $\NN^k,\:k\ge 1$. Then there is a unique natural number $l\in\NN$ such that there is a Presburger definable bijection between $M$ and $\NN^l,$ $1\le l\le k.$
\end{theorem}

\begin{proof}
First let us show that there is some $l$ with the property. According to \sref{Theorem}{fund}, all definable in $(\NN,+)$ sets are disjoint unions of fundamental lattices $L_1,\ldots,L_n$ of the dimensions $s_1,\ldots,s_n$, respectively. It is easy to see that for each $L_i$ there is a linear bijection with $\NN^{s_i}$, which is obviously definable. Let us put $l$ to be the maximum of $s_i$'s.  Now we just need to notice that for each sequence of natural number $r_1,\ldots,r_m$ and $u=\max(r_1,\ldots,r_m)$ if $u\ge 1$ then we could split a set $\NN^u$ into sets $A_1,\ldots,A_m$ for which we have definable bijections with $\NN^{r_1},\ldots,\NN^{r_m}$, respectively. We prove the latter by induction on $m$.

Now let us show that there is no other $l$ with this property. Assume the contrary. Then clearly, for some $l_1>l_2$  there is a mapping $f\colon\NN^{l_1}\ra\NN^{l_2}$. Let us consider a sequence of expanding cubes, $I_n^{l_1}\eqdef\{(x_1\sco x_k)\mid 0\le x_1\sco x_k\le n\}$. We define function $g\colon \NN\to\NN$ to be the function which maps a natural number $n$ to the least $m$ such that $f(I_n^{l_1})\subseteq I_m^{l_2}$. Clearly, $g$ is a Presburger-definable function. Then there should be some linear function $h\colon \NN\to\NN$ such that $g(n)\le h(n)$, for all $n$. But since for each $n\in\NN$ and $m<n^{l_1/l_2}$ the cube $I_n^{l_1}$ contains more points than the cube  $I_m^{l_2},$ from the definition of $g$ we see that $g(n)\ge n^{l_1/l_2}$. This contradicts the linearity of the function $h$.\qed
\end{proof}

From the proof above we see that the following corollary holds:
\begin{corollary} \label{sublattice}
The dimension of a set $M\subseteq\NN^k$ is equal to the maximal $l$ such that there exists an exactly $l$-dimensional fundamental lattice which is a subset of $M.$
\end{corollary}

\subsection{Presburger-Definable Linear Orders}

\begin{lemma}\label{MS}
Let $\overline{x}=(x_1,\ldots,x_n)$ and $\overline{y}=(y_1,\ldots,y_k)$ be vectors of free variables, where $\overline{y}$ will be treated as a vector of parameters. Let $F(\xv,\yv)$ be an $\Lc^{-}$-formula such that for an infinite set of parameter vectors $B=\{\bv_1,\bv_2,\ldots\}$ the sets defined by $F(\xv,\bv_i)$ are disjoint in $\NN^n.$ Then only a finite number of those definable sets can be exactly $n$-dimensional.
\end{lemma}

\begin{proof}
Let us consider the set $A\subseteq \NN^{n+k}$ defined by the formula $F(\xv,\yv)$.  For each vector $\bv=(b_1,\ldots,b_k)\in \mathbb{N}^k$ and set $S\subseteq \NN^{n+k}$ we consider section $S\upharpoonright\bv =\{(a_1,\ldots,a_n,b_1,\ldots,b_k) \mid (a_1,\ldots,a_n,b_1,\ldots, b_k)\in S\}$. Clearly in this terms in order to prove the lemma, we need to show that there are only finitely many distinct $\bv\in B$ such that the section $A\upharpoonright \bv$ is an $n$-dimensional set. By \sref{Theorem}{fund}, the set $A$ is a disjoint union of finitely many of fundamental lattices  $J_i\subseteq \NN^{n+k}.$ It is easy to see that if some section $A\upharpoonright\bv$ were an $n$-dimensional set then at least for one $J_i$, the section $J_i\upharpoonright\bv$ were an $n$-dimensional set. Thus it is enough to show that for each $J_i$ there are only finitely many vectors $\bv\in B$ for which the section $J_i\upharpoonright\bv$ is an $n$-dimensional set.

Let us now assume for a contradiction that for some $J_i$ there are infinitely many $J_i\upharpoonright \bv_0$, for $\bv_0\in B$, that are $n$-dimensional sets. Let us consider some parameter vector $\bv\in\NN^k$ such that the section $J\upharpoonright\bv$ is an $n$-dimensional set. Then by \sref{Corollary}{sublattice} there exists an $n$-dimensional fundamental lattice $K\subseteq J_i\upharpoonright\bv_0$. Suppose the generating vectors of $K$ are $\vv_1,\ldots,\vv_n$ and initial vector of $K$ is $\uv$. It is easy to see that each vector $\vv_j$ is a non-negative linear combination of generating vectors of $J$, since otherwise for large enough $h\in \mathbb{N}$ we would have $\cv+h\vv_j\not\in J$. Now notice that for any  $\bv\in B$ and $\av \in J\upharpoonright \bv$ the $n$-dimensional lattice with generating vectors $\vv_1,\ldots,\vv_n$ and initial vector $\av$ is a subset of $\av \in J\upharpoonright \bv$.

Thus infinitely many of the sets defined by $F(\xv,\bv)$, for $\bv\in B$ contain the shifts of the same $n$-dimensional fundamental lattice. It is easy to see that the latter contradicts the assumption that all the sets are disjoint.\qed
\end{proof}

\begin{definition}
We call a linear ordering $(L,<)$ \emph{scattered} 
if it does not have an infinite dense suborder.
\end{definition}

\begin{definition}
  Let $(L,\prec)$ be a linear ordering. We define a family of equivalence relations $\simeq_{\alpha}$, for ordinals $\alpha\in\mathbf{Ord}$ by transfinite recursion:
  \begin{itemize}
  \item $\simeq_0$ is just equality;
  \item $\simeq_{\lambda}=\bigcup\limits_{\beta<\lambda}\simeq_{\alpha}$, for limit ordinals $\lambda$;
  \item $a\simeq_{\alpha+1}b \stackrel{\mbox{\footnotesize $\mathrm{def}$}}{\iff} |\{c\in L\mid (a\prec c\prec b)\mbox{ or }(b\prec c \prec a)\}/{\simeq_{\alpha}}|<\aleph_0$.
  \end{itemize}
  Let us define  $VD_*$-\emph{rank}\footnote{$VD$ stand for {\em very discrete}; see \cite[p.\,84-89]{rosenstein}.} $\mathrm{rk}(L,\prec)\in \mathbf{Ord}\cup \{\infty\}$ of the order $(L,\prec)$. The $VD_*$-rank $\mathrm{rk}(L,\prec)$ is the least $\alpha$ such that $L/{\simeq_{\alpha}}$ is finite. And if for all $\alpha\in \mathbf{Ord}$ the factor-set $L/{\simeq_{\alpha}}$ is infinite  then we put $\mathrm{rk}(L,\prec)=\infty$.

  By definition we put $\alpha<\infty$, for all $\alpha\in \mathbf{Ord}$.
\end{definition}
\begin{remark}
Linear orders $(L,\prec)$ such that $\mathrm{rk}(L,\prec)<\infty$ are exactly the scattered linear orders.
\end{remark}

\begin{example} \label{rank0rank1}
The orders with the $VD_*$-rank equal to $0$ are exactly finite orders, and the orders with $VD_*$-rank $\le 1$ are exactly the order sums of finitely many copies of $\NN$, $-\NN$ and $1$ (one element linear order).
\end{example}
%


%
%


\begin{theorem}[Restatement of \sref{Theorem}{ordering}]\label{rank_from_dimension}
For every natural $m\ge 1$, linear orders which are $m$-dimensionally interpretable in $(\mathbb{N},+)$ have $VD_*$-rank $m$ or below.
\end{theorem}
\begin{proof}


We prove the theorem by induction on $m$.

Suppose we have an $m$-dimensional interpretation of a linear order $(L,\prec)$ in $(\NN,+)$, i.e. there is an $\Lc^{-}$ formula $D(\xv)$ giving the domain of the interpretation and $\Lc^{-}$ formula $\prec_*(\xv,\yv)$ giving interpretation of the order relation, where both $\xv$ and $\yv$ consist of $m$ variables. Without loss of generality we may assume that $L= \{\av \in \mathbb{N}^m \mid (\mathbb{N},+)\models D(\av)\}$ and $\prec$ is defined by the formula $\prec_*$.

Now assume for a contradiction that $\mathrm{rk}(L,\prec)>m$. By the definition of $VD_*$-rank, there are infinitely many distinct $\simeq_m$-equivalence classes in $L$. Hence there is an infinite chain $\av_0\prec \av_1\prec\ldots $ of elements of $L$ such that $\av_i\not\simeq_m \av_{i+1}$, for each $i$. Let us consider intervals $L_i=\{\bv\in L\mid \av_i<\bv<\av_{i+1}\}$. Since $\av_i\not\simeq_m \av_{i+1}$, the set $L_i/{\simeq_{m-1}}$ is infinite and $\mathrm{rk}(L_i,\prec)>m-1$.

Clearly, all $L_i$ are Presburger definable sets. Let us show that $\dim(L_i)\ge m$, for each $i$. If $m=1$ then it follows from the fact that $L_i$ is infinite. If $m>1$ then we assume for a contradiction that $\dim(L_i)<m$. And notice that in this case $(L_i,\prec)$ would be $m-1$-dimensionally interpretable in $(N,+)$ which contradict induction hypothesis and the fact that $\mathrm{rk}(L_i,\prec)>m-1$. Since $L_i\subseteq \NN^m$, we conclude that $\dim(L_i)=m$, for all $i$.

Now consider the parametric family of subsets of $\NN^m$ given by the formula $\yv_1\prec_* \xv\prec_* \yv_2$, where we treat variables $\yv_1$ and $\yv_2$ as parameters. We consider sets given by pairs of parameters $\yv_1=\av_i$ and $\yv_2=\av_{i+1}$, for $i\in\mathbb{N}$. Clearly the sets are exactly $L_i$'s. Thus we have infinitely many disjoint sets of the dimension $m$ in the family and hence we have contradiction with Lemma \ref{MS}.
\end{proof}

\begin{remark} Each scattered linear order of $VD_*$-rank 1 is $1$-dimensionally interpretable in $(\NN,+)$. There are scattered linear orders of $VD_*$-rank 2 that are not interpretable in $(\NN,+)$.
\end{remark}
\begin{proof}
The interpretability of linear orders with rank $0$ and rank $1$ follows from \sref{Example}{rank0rank1}.

Since there are uncountably many non-isomorphic scattered linear orders of $VD_*$-rank 2 and only countably many linear orders interpretable in $(\NN,+)$, there is some scattered linear order of $VD_*$-rank 2 that is not interpretable in $(\NN,+)$.\qed
\end{proof}

\section{One-Dimensional Self-Interpretations and Visser's Conjecture}

The following theorem is a generalization of \cite[pp.\,27-28,\;Lemmas 3.2.2-3.2.3]{jetze}.

\begin{theorem}\label{4r}
  Let $\mathbf{U}$ be a theory and $\iota$ be an $m$-dimensional interpretation of $\mathbf{U}$ in $(\NN,+)$. Then for some $m'\le m$ there is an $m'$-dimensional non-relative interpretation with absolute equality $\kappa$ of $\mathbf{U}$ in $(\NN,+)$ which is definably isomorphic to $\iota$.
\end{theorem}

\begin{proof}
  First let us find $\kappa$ with absolute equality. Indeed there is a definable in $(\NN,+)$ well-ordering $\prec$ of $\NN^m$:
  $$(a_0,\ldots,a_{m-1})\prec (b_0,\ldots,b_{m-1})\stackrel{\mbox{\footnotesize \textrm{def}}}{\iff} \exists  i< m (\forall j< i\;(a_j=b_j)\land a_i<b_i).$$
  Now we could define $\kappa$ by taking the definition of $+$ from $\iota$, taking the trivial interpretation of equality, and taking the domain of interpretation to be the part of the domain of $\iota$ that consists of the $\prec$-least elements of equivalence classes with respect to $\iota$-interpretation of equality. It is easy to see that this $\kappa$ is definably isomorphic to $\iota$.

  Now assume that we already have $\iota$ with absolute equality. We find the desired non-relative interpretation $\kappa$ by using \sref{Theorem}{Existe} and bijectively mapping the domain of $\iota$ to $\NN^{m'}$, where $m'$ is the dimension of the domain of the interpretation $\iota$.\qed
\end{proof}

Combining \sref{Theorem}{models-class} and \sref{Theorem}{rank_from_dimension}, we obtain
\begin{theorem}[Restatement of \sref{Theorem}{1a}]\label{order}
For any model $\mathfrak{A}$ of $\PrA$ that is one-dimensionally interpreted in the model $(\NN,+)$, (a) $\mathfrak{A}$ is isomorphic to $(\NN,+)$; (b) the isomorphism is definable in $(\NN,+)$.
\end{theorem}
\begin{proof} Let us denote by $<_*$ the order relation given by the $\PrA$ definition of $<$ within $\mathfrak{A}$. Clearly $<_*$ is definable in $(\NN,+)$. Thus we have an interpretation of the order type of $\mathfrak{A}$ in $\PrA$. Hence by \sref{Theorem}{rank_from_dimension} the order type of $\mathfrak{A}$ is scattered. But from \sref{Theorem}{models-class} we know that the only case when the order type of a model of $\PrA$ is scattered is the case when it is exactly $\NN$. Thus $\mathfrak{A}$ is isomorphic to $(\NN,+)$. From \sref{Theorem}{4r} it follows that it is enough to show the definability of the isomorphism only in the case when the interpretation that gives us $\mathfrak{A}$ is a non-relative interpretation with absolute equality.

It is easy to see that, the isomorphism $f$ from $\mathfrak{A}$ to $(\NN,+)$ is the function $f\colon x\longmapsto |\{y\in\mathbb{N}\mid y<_*x\}|$. Now we use counting quantifier to express the function:

\begin{gather}
f(a)=b \iff (\mathbb{N},+)\models \exists^{=b}z \;(z<_*a)
\end{gather}

Now apply \sref{Theorem}{unti} and see that $f$ is definable in $(\mathbb{N},+)$.
\end{proof}

\begin{theorem} Theory $\PrA$ is not one-dimensionally interpretable in any of its finitely axiomatizable subtheories.

\end{theorem}
\begin{proof} Assume $\iota$ is an one-dimensional interpretation of $\PrA$ in some  finitely axiomatizable subtheory $\mathrm{T}$ of $\PrA$. In the standard model $(\NN,+)$ the interpretation $\iota$ will give us a model $\mathfrak{A}$ for which there is a definable isomorphism $f$ with $(\NN,+)$. Now let us consider theory $\mathrm{T}'$ that consists of $\mathrm{T}$ and the statement that the definition of $f$ gives an isomorphism between (internal) natural numbers and the structure given by $\iota$. Clearly $\mathrm{T}'$ is finitely axiomatizable and true in $(\NN,+)$, and hence is subtheory of $\PrA$. But now note that $\mathrm{T}'$ proves that if something was true in the internal structure given by $\iota$, it is true. And since $\mathrm{T}'$ proved any axiom of $\PrA$ in the internal structure given by $\iota$, the theory $\mathrm{T}'$ proves every axiom of $\PrA$. Thus $\mathrm{T}'$ coincides with $\PrA$. But it is known that $\PrA$ is not finitely axiomatizable, contradiction.
\end{proof}

\section{Multi-Dimensional Self-Interpretations}

We already know that the only linear orders that it is possible to interpret in $(\NN,+)$ (even by multi-dimensional interpretations) are scattered linear orders. And we could use this to prove the analogue of  \sref{Theorem}{1a}(a) for multi-dimensional interpretations by the same reasoning as we have used for \sref{Theorem}{1a}(a).

However, the only way any interpretation can be isomorphic to trivial in a multi-dimensional case is by having a one-dimensional set as its domain and from \sref{Theorem}{1a} it follows that all interpretations of $\PrA$ in $(\mathbb{N},+)$ that have one-dimensional domain are definably isomorphic to $(\mathbb{N},+)$. Thus in order to prove the analogue of \sref{Theorem}{1a}(b) for multi-dimensional interpretations one should in fact show that the domain of any interpretation of $\PrA$ in $(\NN,+)$ should be one-dimensional set.


In the section we will give some partial results about multi-dimensional self-interpretations of $\PrA$.

{\em Cantor polynomials} are quadratic polynomials that define a bijection between $\NN^2$ and $\NN:$
\begin{gather}
C_1(x,y)=C_2(y,x)=\frac{1}{2}(x+y)^2+\frac{1}{2}(x+3y).\label{cantorps}
\end{gather}

The bijections $C_1$ and $C_2$ are the isomorphism of $(\NN^2,\prec_1)$ and $(\NN,<)$ and the isomorphism of $(\NN^2,\prec_2)$ and $(\NN,<)$, where
\begin{gather*}
  (a_1,a_2)\prec_1(b_1,b_2)\stackrel{\mbox{\footnotesize \textrm{def}}}{\iff}(a_2<b_2\wedge{a_1+a_2=b_1+b_2})\vee(a_1+a_2<b_1+b_2),\\
  (a_1,a_2)\prec_2(b_1,b_2)\stackrel{\mbox{\footnotesize \textrm{def}}}{\iff}(a_2>b_2\wedge{a_1+a_2=b_1+b_2})\vee(a_1+a_2<b_1+b_2).
\end{gather*}
Note that both $\prec_1$ and $\prec_2$ are definable in $(\NN,+)$. The following theorem show that this interpretations of $(\NN,<)$ could not be extended to interpretations of $(\NN,x\mapsto sx)$, for some $s$ and thus shows that this interpretations could not be extended to interpretations of $(\NN,+)$.

\begin{theorem}\label{en}
 Let $s$ be a natural number that is not a square and $i$ be either $1$ or $2$. Let us denote by $f\colon \NN^2\to \NN^2$ the function $f(\overline{a})=C_i^{-1}(s\cdot C_i(\overline{a}))$, i.e. the preimage of the function $x\mapsto s\cdot x$ under the bijection $C_i\colon \NN^2\to \NN$. Then the function $f$ is not definable in $(\NN,+)$.
\end{theorem}
\begin{proof} Since the cases of $i=1$ and $i=2$ are essentially the same, let us consider just the case of $i=1$. Suppose the contrary: there is an $\Lc^{-}$-formula $F(x_1,x_2,y_1,y_2)$ which defines the graph of $f$:
$$(\NN,+)\models F(a_1,a_2,b_1,b_2)\iff f(a_1,a_2)=(b_1,b_2), \mbox{for all $a_1,a_2,b_1,b_2\in \NN$.}$$
Then  the following function $h(x):\NN\ra\NN$  is also definable:
\begin{gather}
  h(a)=b\stackrel{\mbox{\footnotesize \textrm{def}}}{\iff}\exists c,d ( f(a,0)=(c,d)\land b=c+d).
\end{gather}
Now it is easy to see that the following inequalities holds for all $a\in\NN$:
\begin{gather*}
  C_1(h(a),0)\le s\cdot C_1(a,0)<C_1(h(a)+1,0)\Ra\\
  \frac{h(a)(h(a)+1)}{2}\le \frac{sa(a+1)}{2}<\frac{(h(a)+1)(h(a)+2)}{2}\Ra\\
  y^2<S(x+1)^2\mbox{ and }Sx^<(y+2)^2 \Ra\\
  \sqrt{S}x-2<y<\sqrt{S}x+\sqrt{S}.
\end{gather*}

We conclude that a Presburger-definable function $h(x)$ is bounded both from above and below with linear functions of the same irrational slope. Contradiction with \sref{Corollary}{irration}.\qed
\end{proof}



We conjecture the following general fact holds:
\begin{hyp}\label{ba}
For any (multi-dimensional) interpretation $\iota$ of $\PrA$ in the model $(\NN,+)$ there is a definable isomorphism with the trivial interpretation of $(\NN,+)$ in $(\NN,+)$.
\end{hyp}

%
%


The following theorem is a slight modification of the theorem by G.R.~Blakley~\cite{blakley}.
\begin{theorem}\label{bash}
Let $A$ be a $d\times n$ matrix of integer numbers, function $\varphi_A\colon\mathbb{Z}^d\ra \NN\cup \{\aleph_0\}$ is defined as follows:

\begin{center}
$\varphi_A(u)\eqdef|\{\overline{\lambda}=(\lambda_1,\ldots,\lambda_n)\in\NN^n\mid A\lambda=u\}|.$
\end{center}

Then if the values of $\varphi_A$ are always finite, the function $\varphi_A$ is a piecewise polynomial function of a degree $\le n-\mathrm{rk}(A)$.
\end{theorem}

\begin{proof}
The existence of the fundamental lattices $C_1,\ldots,C_l$ on which $\varphi_A$ is polynomial follows from \cite[p.\,302]{sturmfels}. Now we prove that the $n-\mathrm{rk}(A)$ bound on the degree holds.

Let us consider any fundamental lattice $L$ with the initial vector $\overline{v}$ and generating vectors $\overline{s}_1,\ldots,\overline{s}_m$ such that the restriction of  $\varphi_A$ to $L$ is a polynomial. Now it is easy to see that we could find a polynomial $P(x_1,\ldots,x_m)$ such that $\varphi_A(\overline{v}+\eta_1\overline{s}_1+\ldots+\eta_m\overline{s}_m)=P(\eta_1,\ldots,\eta_m)$, for all $\eta_1,\ldots,\eta_m\in \NN$. Since the choice of $L$ was arbitrary, we could finish the proof of the theorem by showing that $P$ is of the degree $\le n-\mathrm{rk}(A)$. Let us assume for a contradiction that the degree of $P$ is $>n-\mathrm{rk}(A)$. Clearly, then there are $\theta_1,\ldots,\theta_m\in \NN$ such that the polynomial $Q(y)=P(\theta_1y,\ldots,\theta_m y)$ is of the degree $k>n-\mathrm{rk}(A)$. Now we consider the vector $\overline{d}=\eta_1\overline{s}_1+\ldots+\eta_m\overline{s}_m$ and the vectors $\overline{e}_l=\overline{v}+l\overline{d}$, for $l\in \mathbb{N}$. We have $\varphi_A(\overline{e}_l)=Q(l)$.

Let us now estimate the values of $\varphi_A(\overline{e}_l)$. The value $\varphi_A(\overline{e}_l)$ is the number of integer points in the polyhedron $H_l=\{(\lambda_1,\ldots,\lambda_n)=\overline{\lambda}\in \mathbb{R}^n\mid A\overline{\lambda}=\overline{e}_l\mbox{ and }\lambda_1,\ldots,\lambda_n\ge 0\}$. And now it is easy to see that $\varphi_A(\overline{e}_l)\le h_l/o$, where $o$ is the volume of $(n-\mathrm{rk}(A))$-dimensional sphere of the radius $1/2$ and $h_l$ is the $(n-\mathrm{rk}(A))$-dimensional volume of (at most) $(n-\mathrm{rk}(A))$-dimensional polyhedron $H_l'=\{(\lambda_1,\ldots,\lambda_n)=\overline{\lambda}\in \mathbb{R}^n\mid A\overline{\lambda}=\overline{e}_l\mbox{ and }\lambda_1,\ldots,\lambda_n\ge -1\}$. Now we just need to notice that the linear dimensions of the polyhedra $H_l'$ are bounded by a linear function of $l$ and hence the volumes $h_l$ are bounded by some polynomial of the degree $n-\mathrm{rk}(A)$, contradiction with the fact that the polynomial $Q(y)$ were of the degree $k>n-\mathrm{rk}(A)$.\qed
\end{proof}

Recall that a semilinear set is a finite union of lattices and that by result of \cite{ir} any semilinear set is a disjoint union of fundamental lattices. It is easy to see that the following lemma holds:

\begin{lemma} \label{pwpol_prop}
\begin{enumerate}
\item \label{pwpol_add} If $f,g\colon A\to \mathbb{Z}$ are piecewise polynomial functions of a degree $\le m$ then the function $h\colon A\to \mathbb{Z},\;h(\vv)=f(\vv)+g(\vv)$, is a piecewise polynomial function of a degree $\le m$;
\item \label{pwpol_rest} if $A\subseteq \mathbb{Z}^n$ is a semilinear set, $f\colon A\to \mathbb{Z}$ is a piecewise polynomial function of a degree $\le m$, and $B\subseteq A$ is $\PrA$-definable set then the restriction of $f$ to $B$ is a piecewise polynomial function of a  degree $\le m$;
\item \label{pwpol_lin} if $A\subseteq \mathbb{Z}^n$ is a semilinear set, $f\colon A\to \mathbb{Z}$ is a piecewise polynomial function of a degree $\le m$, and $F\colon \mathbb{Z}^n\to \mathbb{Z}^k$ is a linear operator, then the function $h\colon F(A)\to\mathbb{Z}^k$ is a piecewise polynomial function of a  degree $\le m$.
\end{enumerate}
\end{lemma}

We prove the lemma that generalizes the one-dimensional construction of the cardinality of sections.

\begin{lemma}\label{newlem}
Let $S\subseteq\NN^{n+m}$ be a definable set in $(\NN,+)$. For each vector $\overline{b}=(b_1,\ldots,b_m)\in \NN^m$ we define section $A\upharpoonright {\overline{b}}$ to be the set of all elements of $S$ of the form $(a_1,\ldots,a_n,b_1,\ldots,b_m)$. Suppose all sets $S\upharpoonright \overline{b}$ are finite. For each vector $\overline{a}\in\NN^n$. Consider the section cardinality function $f_S\colon\NN^m\ra\NN,\:f_S\colon\overline{a}\mapsto|S\upharpoonright \overline{b}|.$ Then $f_S$ is a piecewise polynomial function of a  degree $\le n$.
\end{lemma}

\begin{proof} Let us first prove the theorem for the case when $S$ is a fundamental lattice with the initial vector $\cv$ and the generating vectors $\vv_1,\ldots,\vv_s\in \NN^{n+m}$. We consider the vectors $\cv',\vv_1',\ldots,\vv_s'\in \NN^m$ that consist of the last $m$ components of vectors $\cv,\vv_1,\ldots,\vv_s$, respectively. Clearly, for each $\bv\in \NN^m$, the value $f_S(\bv)=|A\upharpoonright \bv|$ is equal to the number of different $\overline{\lambda}=(\lambda_1,\ldots,\lambda_s)\in\NN^s$ such that $\lambda_1\vv_1'+\ldots+\lambda_s\vv_s'=\bv-\cv'$. Now we compose a matrix $A$ from the vectors $\vv_1',\ldots,\vv_s'$ and see that $f_S(\bv)=|\{\overline{\lambda}\in \NN^m\mid A \overline{\lambda}=\bv-\cv'\}|=\phi_{A}(\bv-\cv)$. Note that since $S$ was a fundamental lattice, $s-\mathrm{rk}(A)\le n$. Now we apply \sref{Theorem}{bash} and see that $\phi_{A}$ is a piecewise polynomial of a degree $\le n$. Now from \sref{Lemma}{pwpol_prop}(\ref{pwpol_rest}) and \sref{Lemma}{pwpol_prop}(\ref{pwpol_lin}) it follows that $f$ is piecewise polynomial of a degree $\le n$ too.

In the case of arbitrary definable $A$, we apply \sref{Theorem}{fund} and find fundamental lattices $J_1,\ldots,J_s$ such that $A=J_1\sqcup J_2\sqcup\ldots\sqcup J_s$. Now we see that for each $\overline{b}\in\NN^m$, we have $f_A(a)=f_{J_1}(a)+\ldots+f_{J_s}(a)$ and since we already know that all $f_{J_i}$ are piecewise polynomial of a degree $\le n$, by \sref{Lemma}{pwpol_prop}(\ref{pwpol_add}) the function $f_A$ is piecewise polynomial of a degree $\le n$.\qed

\end{proof}

\begin{theorem} Suppose a definable in $(\NN,+)$ binary relation $\prec$ on $\NN^n$ has the order type $\NN$. Then the order isomorphism between $(\NN^m,\prec)$ and $(\NN,<)$ is a piecewise polynomial function of a degree $\le n$.
\end{theorem}
\begin{proof}
We see that the order isomorphism is the function $f\colon \NN^m\to \NN$ given by

\begin{center}
$f(a_1,\ldots,a_n)=|\{(b_1,\ldots,b_n,a_1,\ldots,a_n)\mid (b_1,\ldots,b_n)\mathrel{R} (a_1,\ldots,a_n)\}|.$
\end{center}

By \sref{Lemma}{newlem} we see that $f$ is a piecewise polynomial function. \qed
\end{proof}






Fueter-P\'olya theorem \cite{fueterpolya,nathanson}  states that every quadratic polynomial that maps $\NN^2$ onto $\NN$ is one of two Cantor polynomials (\ref{cantorps}). If one would want to prove \sref{Conjecture}{ba} one of the possible approaches would be to give a classification of all piecewise polynomial bijections and then use the classification and a generalization of \sref{Theorem}{en} in order to show that no two-dimensional non-relative interpretation of $(\NN,<)$ in $(\NN,+)$ could be extended to an interpretation of $(\NN,+)$.

\section*{Acknowledgments}

The authors wish to thank Lev~Beklemishev for suggesting to study Visser's conjecture, number of discussions of the subject, and his useful comments on the paper.

\end{document}